\documentclass[fullpage,11pt]{article}

\usepackage[utf8]{inputenc}
\usepackage{amssymb,amsmath,amsthm,color}
\usepackage{hyperref}
\usepackage{cleveref}
\usepackage{graphicx}
\usepackage{geometry}

\newcommand{\R}{\mathbb{R}}
\newcommand{\PP}{\mathbb{P}}

\newcommand{\EE}{\mathbb{E}}
\newcommand{\LCI}{\mathrm{LCI}}
\newcommand{\PD}{\mathrm{PD}}
\providecommand{\dist}{\operatorname{dist}}
\providecommand{\aff}{\operatorname{aff}}

\newtheoremstyle{break}
  {\topsep}{\topsep}
  {\itshape}{}
  {\bfseries}{}
  {\newline}{}
\theoremstyle{break}
\newtheorem{theorem}{Theorem}
\newtheorem*{theorem*}{Theorem}
\newtheorem{lemma}[theorem]{Lemma}
\newtheorem{observation}[theorem]{Observation}
\newtheorem{remark}[theorem]{Remark}

\newtheorem{claim}[theorem]{Claim}

\newtheorem{proposition}[theorem]{Proposition}
\newtheorem{fact}[theorem]{Fact}
\newtheorem{corollary}[theorem]{Corollary}
\theoremstyle{definition}
\newtheorem*{definition}{Definition}
\theoremstyle{break}

\title{Determining a Point Configuration from a Subset of the Pairwise Distances}

\author{Itai Benjamini \thanks{\href{mailto:itai.benjamini@weizmann.ac.il}{itai.benjamini@weizmann.ac.il}}
\and Elad Tzalik \thanks{\href{mailto:elad.tzalik@weizmann.ac.il}{elad.tzalik@weizmann.ac.il}}
}
\date{
Weizmann Institute of Science \\[2ex]
\today
}

\begin{document}

\maketitle

\begin{abstract}
    We study rigidity without assuming general position. Given \(n\) distinct labelled points and a set \(\mathcal{P}\subseteq \binom{[n]}{2}\) of revealed pairs, we ask when the corresponding distances determine the configuration up to isometry. On the line, we prove an extremal result: if \(|\mathcal{P}|=\Omega(n^{3/2})\), then there is an induced globally rigid subgraph on \(\Omega(|\mathcal{P}|/n)\) vertices. In other words, any dense enough graph will contain a subset of labels whose locations can be determined from their distances up to isometry. To prove this, we establish a graph-theoretic result, which may be of independent interest: a dense graph in which every non-edge has few common neighbours contains a clique of size \(\Omega(|E|/n)\).

We also study random revealed pairs. For every labelled configuration \(V\) of distinct points in \(\mathbb R\), if each pair is revealed independently with probability \(p=C\ln n/n\), where \(C>1\), then the revealed distances determine \(V\) w.h.p. We prove a similar result for \(d\ge1\) under the mild non-degeneracy assumption that every subcollection of more than \(\tau n\) points of \(V\subseteq\mathbb R^d\) affinely spans \(\mathbb R^d\), for some fixed \(0<\tau<1\). In this case, every \(C>1/(1-\tau)\) suffices. The same ideas also settle the weak-threshold form of a conjecture of Girão et al. for a giant reconstructable component, and substantially improve in this direction the work of Barnes et al. establishing such a component for \(p>n^{-2/(d+4)}\).

Subsequent works have sharpened the random setting. Girão, Illingworth, Michel, Powierski, and Scott proved a hitting time result for the first moment at which one can reconstruct \(V\) when \(\mathcal{P}\) is revealed using the Erdős--Rényi evolution; our extremal result is central to their argument. Montgomery, Nenadov, Portier and Szabó resolved one of our conjectures and proved that w.h.p. a graph sampled from the Erdős--Rényi evolution becomes globally rigid in \(\mathbb{R}\) at the moment its minimum degree is \(2\).

\end{abstract}

\section{Introduction}

We consider the following question: given $n$ distinct points $x_1,\ldots,x_n$ in a metric space $M$, and $\mathcal{P} \subseteq \binom{[n]}{2}$, are the points determined up to isometry from their partial distances according to $\mathcal{P}$? In $\R^d$ with $d>1$, it is easy to construct examples where all pairwise distances except one are given, yet the distance of the missing pair cannot be deduced from the others \footnote{In other words there exist $x_1,\ldots,x_n$ and $y_1,\ldots,y_n$ that are not isometric, yet $\dist(x_i,x_j)=\dist(y_i,y_j)$ for all $\{i,j\}\neq\{1,2\}$, $i,j \in [n]$.}. In order to say something meaningful on the problem above, a popular line of work in rigidity theory is to assume $x_1,\ldots, x_n$ are generic, meaning that all their coordinates are algebraically independent over $\mathbb{Q}$. In that case, one can show that some graphs can distinguish all $n$-tuples of points in generic position by their pairwise distances, while for the rest of the graphs there will be two non-isometric $n$-tuples that have the same pairwise distances.

In $\R$, it is well-known that if the points are in generic position, then all distances are determined by $\mathcal{P}$ if and only if $([n],\mathcal{P})$ is $2$-connected, as shown in \cite{global_rigidity_chap}. However, when the points are not in generic position, it is easy to construct two $4$-tuples that the $4$-cycle, $C_4$, cannot distinguish.

We study rigidity \emph{without} assuming general position but only assuming the weaker assumption that the given points are distinct, and show that even under this weaker assumption, the distance information along a graph $([n],\mathcal{P})$ carries meaningful information.

\begin{definition}
For a graph $G=([n],\mathcal{P})$ and an injective map $f \colon [n] \to M$, let $d_{G,f} \colon \mathcal{P} \to \mathbb{R}_{\geq 0}$ be the edge-distance function given by $d_{G,f}(\{i,j\})=\dist(f(i),f(j))$. We say that $f$ realizes a function $d \colon \mathcal{P}\to \mathbb{R}_{\geq 0}$ if $d=d_{G,f}$. A graph $G$ is \emph{globally rigid} in $M$ if, for any two injective maps $f,g \colon [n]\to M$ such that $d_{G,f}=d_{G,g}$, one has $\dist(f(i),f(j))=\dist(g(i),g(j))$ for all $\{i,j\}\in \binom{[n]}{2}$.
\end{definition}

\begin{definition}
For a graph $G=([n],\mathcal{P})$, an injective map $f \colon [n]\to M$, and $U\subseteq [n]$, we say that $U$ is \emph{reconstructable} from $(G,d_{G,f})$ if every injective map $g\colon [n]\to M$ which realizes $d_{G,f}$ satisfies $\dist(f(i),f(j))=\dist(g(i),g(j))$ for all $\{i,j\}\in \binom{U}{2}$. A pair $\{i,j\}$ is reconstructable if the set $\{i,j\}$ is reconstructable.
\end{definition}

The key distinction is the order of quantifiers. Global rigidity asks for a fixed graph \(G\) whose edge lengths determine the full distance matrix for every injective embedding. Reconstructability is a property of \(G\) together with a particular embedding: the same graph may reconstruct all distances for one embedding and fail to do so for another.

This is the main difference from the standard generic notion of global rigidity, where the embedding is assumed to be generic. A characterization of globally rigid graphs in $\R$ in the injective setting was obtained by~\cite[Theorem~2.4]{GARAMVOLGYI2022112687}. They established that certain edge colorings of the graph imply that the graph is \emph{not} globally rigid in $\R$, and showed that recognizing whether such an obstruction exists is NP-complete.

Our first result is extremal: every sufficiently dense graph contains a large induced globally rigid subgraph.

\begin{theorem}\label{THM : Worst-Case-Bound-Line}
There exists an absolute constant $C>0$ such that every graph $G=([n],\mathcal{P})$ with $|\mathcal{P}| \geq C n^{3/2}$ contains an induced subgraph on  $\Omega\left(\frac{|\mathcal{P}|}{n}\right)$ vertices which is globally rigid in $\mathbb R$.
\end{theorem}

For \(S^1\), the same closure proof gives a reconstructable subset of size
\(\Omega(|\mathcal P|/n)\), with a different absolute constant. A step in the proof is the following graph-theoretic result, which may be of independent
interest.

\begin{theorem}\label{Theorem : Graphs with bounded intersection}
    Let $k\geq1$. Let $G=([n],E)$ be a graph, and let $\Gamma(v)$ denote the set of neighbours of $v$. Suppose that every non-edge $\{i,j\}\notin E$ satisfies $|\Gamma(i) \cap \Gamma(j)| \leq k$. If $|E| \geq 8 n \sqrt{kn}$ then $G$ contains a clique of size at least $\frac{1}{4}\frac{|E|}{n}$.
\end{theorem}

We briefly explain the connection between the two theorems. Starting from
\(G=([n],\mathcal P)\), we add every edge whose endpoints have three common
neighbours, and repeat this operation until it stops. On the line, if we know the three
distances to the common neighbours, the injectivity of the mapping implies that the missing distance is determined. Thus, as far as rigidity is concerned, we may look at the graph obtained by repeatedly adding edges between pairs of \emph{non-adjacent} vertices with at least three common neighbours. We call this graph the $K_{2,3}$ closure of $G$.
Theorem \ref{Theorem : Graphs with bounded intersection} gives a large clique in
the closure which translates back to a globally rigid subgraph of the original graph.

We also study what can be deduced from a random $\mathcal{P}$ in higher dimensions. This direction is inspired by the recent work of Lew, Nevo, Peled, and Raz \cite{https://doi.org/10.48550/arxiv.2202.09917}, who obtained a hitting time result for global rigidity in $\mathbb{R}^d$ in the generic setting.

For non-generic configurations, there is an obvious obstruction: if almost all points lie on an affine hyperplane, then points outside this subspace can be reflected along the subspace - the existence of these reflections rules out the possibility of revealing the distances between those points. In order to say something meaningful, we assume the injective map $f$ is fixed in advance - meaning that the embedding is chosen before the randomness is revealed. We show that in this setup the obstruction of many points on an affine subspace is essentially the only obstruction to reconstruction.

\begin{theorem}\label{THM - Random Reconstruction}
Fix \(d\ge1\), \(0<\tau<1\), and \(C>1/(1-\tau)\). Let
\(V=(v_1,\ldots,v_n)\) be a labelled configuration of distinct points in \(\R^d\) such that every
subcollection of more than \(\tau n\) points affinely spans \(\R^d\). If each
pairwise distance is revealed independently with probability \(p=C\ln n/n\),
then w.h.p. the revealed distances determine \(V\) up to Euclidean isometry.
For \(d=1\), this gives the conclusion for every \(C>1\).
\end{theorem}

For the proof of this theorem, the strategy is to show that a random graph distinguishes \(V\) from every non-isometric configuration \(U\). We split the possible \(U\)'s into
two types. First we consider those whose distances disagree with \(V\) on a
positive fraction of the pairs. These are controlled by showing that the
family of all disagreement graphs between \(V\) and all possible \(U\) has
VC-dimension \(O(dn)\). We then apply the epsilon-net theorem, which implies that sampling with probability
\(\Theta(d/n)\) already hits every such disagreement graph, and hence rules
out these \(U\)'s. For the remaining \(U\)'s, we show that there is an isometry $T$ such that $|\{i:T u_i\ne v_i\}|$ is of small linear size in $n$. We show that w.h.p. every remaining index $i$ with $T u_i\ne v_i$ connects to $d+1$ affinely independent points\footnote{This is where the \(\tau\)-spanning property is used} for which $T(U)$ and $V$ agree. We remark that the above theorem is optimal, in the sense that one cannot replace $1/(1-\tau)$ by any smaller constant. We also adapt the method to the setting where the set is not $\tau$-spanning, in which case $V$ can't be reconstructed, but the same methods give a giant component result, confirming a conjecture of Girão, Illingworth, Michel, Powierski and Scott \cite{FUWgirao2023reconstructing}.

\begin{theorem}\label{thm:weak-threshold-linear-reconstruction}
Fix \(d\ge1\). Let \(V=(v_1,\ldots,v_n)\) be any labelled configuration of
distinct points in \(\mathbb R^d\), and let \(G\sim G(n,p)\). Then \(p=1/n\)
is a weak threshold for reconstructing a giant component of \(V\). More
precisely:
\begin{itemize}
    \item if \(pn\to\infty\), then w.h.p. there is a reconstructable set
    \(A\subseteq[n]\) with \(|A|=n-o(n)\);
    \item if \(pn\to0\), then w.h.p. every reconstructable set has size
    \(o(n)\).
\end{itemize}
The statements hold uniformly over \(V\) \footnote{By this we mean that the constants are independent of $V$, though they do depend on $d$.}.
\end{theorem}

For the giant-component result, we use the same two ingredients as in
Theorem \ref{THM - Random Reconstruction}. The main change is that whenever the spanning assumption does not hold we can pass to an affine subspace and induct.

In Appendix \ref{app:monotone-threshold} we record a related result which was used in a previous version of this work but is no longer used in the current proof, yet we
keep it as it seems interesting in its own right. It proves a sharp threshold for the existence of a monotone path in a random graph whose vertex set is $[n]$: the threshold for a monotone path from \(1\) to \(n\) is \(p=\ln n/n\).
Sharper results, including the critical window, were later obtained by
Blanchard, Curien, Krause, and Reisach \cite{blanchard2025phase}.

\paragraph{Subsequent work.} In the one-dimensional random reconstruction direction, Girão, Illingworth, Michel, Powierski, and Scott \cite{FUWgirao2023reconstructing} proved a sharp random reconstruction and hitting-time theorem for reconstructing a fixed point set on the line. Theorem \ref{thm:random-reconstruction-tau-spanning} did not appear in a previous version and their result improved upon ours. Even in the new form, their analysis is much finer, and gives the exact hitting time, which may be before the minimum degree reaches $2$, due to the injectivity assumption. Montgomery, Nenadov, Portier and Szabó \cite{FUWmontgomery2024global} proved that the random graph process becomes globally rigid in \(\mathbb R\) at the moment the minimum degree becomes \(2\). Their results also extend to random regular graphs. Portier \cite{portier2026reconstructing} and Zakharov \cite{zakharov2026sharp} studied the sparse one-dimensional giant reconstructable subset regime. These later results sharpen the one-dimensional random picture; we highlight that our current proof can't yield a giant reconstructable component at $p$ as small as they get, due to losses in the application of Warren sign bound, and the epsilon-net theorem.

Barnes, Petr, Portier, Randall Shaw, and Sergeev \cite{FUWbarnes2024reconstructing} studied the giant-reconstruction in $\mathbb{R}^d$ question posed by Girão et al. and obtained quantitative bounds which we improve upon. Portier \cite{portier2026reconstructing} further studied removing the density assumption in Theorem \ref{THM : Worst-Case-Bound-Line}. Their methods essentially prove a conjecture from a previous version. The density assumption can be completely removed at a cost of a logarithmic loss, and a globally rigid subgraph on $\Omega(|\mathcal P|/(n\log n))$ vertices can be found.

\section{The Extremal Setting}

\subsection{The \(K_{2,r}\)-closure and metric consequences}

Let \(r\geq1\). A \(K_{2,r}\)-configuration in a graph \(F\) is a tuple \((x,y;z_1,\ldots,z_r)\) of distinct vertices such that \(xy\notin E(F)\), while \(xz_a,yz_a\in E(F)\) for every \(a\in[r]\). We call \(x,y\) the left vertices and \(z_1,\ldots,z_r\) the right vertices.

For a graph \(G=([n],E)\), start with \(G^{(0)}=G\). Having defined \(G^{(m)}\), obtain \(G^{(m+1)}\) by adding every non-edge \(xy\) whose endpoints have at least \(r\) common neighbours in \(G^{(m)}\). Since \(G\) is finite, this sequence is eventually constant. We denote the final graph by \(\operatorname{cl}_r(G)\), and call it the \(K_{2,r}\)-closure of \(G\). Thus \(\operatorname{cl}_r(G)\) is \(r\)-closed: every non-edge has at most \(r-1\) common neighbours.

A closure sequence for \(\operatorname{cl}_r(G)\) is an ordering \(e_1,\ldots,e_m\) of the edges in \(E(\operatorname{cl}_r(G))\setminus E(G)\), together with a choice of \(r\) common neighbours for each added edge, such that if \(G_s=G+\{e_1,\ldots,e_s\}\) and \(e_s=x_sy_s\), then \(x_s\) and \(y_s\) have the chosen common neighbours in \(G_{s-1}\). Ordering the new edges by the round in which they first appear, and arbitrarily inside each round, gives such a sequence.

\begin{lemma}[Three common neighbours on the line]\label{lem:three-common-neighbours-line}
Let \(X\) be a finite set, let \(f\colon X\to\mathbb R\) be injective, and let \(i,j,z_1,z_2,z_3\in X\) be distinct. Then \(\dist(f(i),f(j))\) is determined by the six numbers \(\dist(f(i),f(z_t))\) and \(\dist(f(j),f(z_t))\), where \(t=1,2,3\).
\end{lemma}

\begin{proof}
    Write \(a_t=\dist(f(i),f(z_t))\) and \(b_t=\dist(f(j),f(z_t))\).
    Since \(f\) is injective, the three points \(f(z_1),f(z_2),f(z_3)\) are distinct.

    For \(z\in \mathbb R\), set \(h(z)=|\dist(f(i),z)-\dist(z,f(j))|\). There are two cases: outside the interval between \(f(i)\) and \(f(j)\), \(h(z)=\dist(f(i),f(j))\); inside the interval, each smaller value is attained at most twice. Thus the following deterministic rule recovers \(\dist(f(i),f(j))\) from the numbers \(a_t,b_t\): if all three values \(|a_t-b_t|\) are equal, output this common value; otherwise choose \(t\) minimizing \(|a_t-b_t|\) and output \(a_t+b_t\).
\end{proof}

\begin{proposition}\label{prop:line-closure-distances}
Let \(G=([n],E)\), and let \(H=\operatorname{cl}_3(G)\). Then for every injective map \(f\colon[n]\to\mathbb R\), the edge-distance function \(d_{H,f}\) is uniquely determined by \(d_{G,f}\). Equivalently, if \(g\colon[n]\to\mathbb R\) is injective and \(d_{G,g}=d_{G,f}\), then \(d_{H,g}=d_{H,f}\).
\end{proposition}

\begin{proof}
Fix an injective \(g\colon[n]\to\mathbb R\) with \(d_{G,g}=d_{G,f}\). Let \(e_1,\ldots,e_m\) be a closure sequence, and write \(G_s=G+\{e_1,\ldots,e_s\}\). We prove by induction on \(s\) that \(d_{G_s,g}=d_{G_s,f}\). The case \(s=0\) is immediate. Suppose \(e_s=x_sy_s\). By the sequence, \(x_s,y_s\) have three chosen common neighbours \(z_{s,1},z_{s,2},z_{s,3}\) in \(G_{s-1}\). By induction, the six distances from \(x_s,y_s\) to these vertices agree under \(f\) and \(g\). The deterministic rule in Lemma \ref{lem:three-common-neighbours-line} therefore gives \(\dist(f(x_s),f(y_s))=\dist(g(x_s),g(y_s))\).
\end{proof}

\begin{remark}\label{rem:circle-closure}
The same proof works for \(S^1=\mathbb R/\mathbb Z\), with the metric \(\dist(x,y)=\min_{m\in\mathbb Z}|x-y-m|\), after replacing \(\operatorname{cl}_3(G)\) by \(\operatorname{cl}_5(G)\). The required local fact is the circle analogue of Lemma \ref{lem:three-common-neighbours-line}: five common neighbours whose incident distances are already determined determine the distance between the two remaining vertices. Indeed, if \(D=\dist(x,y)\) and \(h(z)=|\dist(x,z)-\dist(z,y)|\), then \(h=D\) on two arcs and every other value occurs at most four times. Thus, given five distinct right vertices \(z_t\), if all values \(|a_t-b_t|\) are equal, their common value is \(D\); otherwise choose \(t\) minimizing \(|a_t-b_t|\), and then \(D=\min\{a_t+b_t,1-a_t-b_t\}\), where \(a_t=\dist(x,z_t)\) and \(b_t=\dist(z_t,y)\). Consequently, if \(H=\operatorname{cl}_5(G)\) and \(f\colon[n]\to S^1\) is injective, then \(d_{H,f}\) is determined by \(d_{G,f}\).
\end{remark}

We now turn to finding rigid subgraphs in the closure, and show they give rise to such subgraphs in $G$. Fix a closure sequence for \(\operatorname{cl}_3(G)\). For each edge \(e\in E(\operatorname{cl}_3(G))\), we define a set \(U(e)\subseteq[n]\), the vertices used to obtain \(e\), recursively along the sequence. If \(e=xy\in E(G)\), set \(U(e)=\{x,y\}\). If \(e=xy\) is added using common neighbours \(z_1,z_2,z_3\), set
\[
    U(e)
    =
    \{x,y,z_1,z_2,z_3\}
    \cup
    \bigcup_{a=1}^3 U(xz_a)
    \cup
    \bigcup_{a=1}^3 U(yz_a).
\]
At the step where \(e\) is added, the edges \(xz_a\) and \(yz_a\) belong to the previous graph, so their sets \(U(xz_a)\) and \(U(yz_a)\) have already been defined. If \(Q\) is a clique in \(\operatorname{cl}_3(G)\), define
\(U(Q)=Q\cup\bigcup_{xy\in\binom{Q}{2}} U(xy)\).

\begin{lemma}\label{lem:support-globally-rigid}
Let \(e\in E(\operatorname{cl}_3(G))\). Then \(G[U(e)]\) is globally rigid in \(\mathbb R\).
\end{lemma}

\begin{proof}
We induct on the time at which \(e\) appears in the closure sequence, with the original edges treated as time \(0\). If \(e=xy\in E(G)\), then \(G[U(e)]\) is a single edge, hence globally rigid.

Suppose \(e=xy\) is added using common neighbours \(z_1,z_2,z_3\). Let \(f,g\colon U(e)\to\mathbb R\) be injective realizations of \(G[U(e)]\) with the same edge lengths. Their restrictions to any subset of \(U(e)\) are still injective. By induction, each graph \(G[U(xz_a)]\) and \(G[U(yz_a)]\) is globally rigid. These are induced subgraphs of \(G[U(e)]\), so the restrictions of \(f\) and \(g\) to them have the same edge lengths. Hence the six distances \(xz_a,yz_a\) are the same in the two realizations. By Lemma \ref{lem:three-common-neighbours-line}, the distance \(xy\) is also the same for \(f\) and \(g\).

Since the distance \(xy\) agrees in the two realizations, compose \(g\) with an isometry of \(\mathbb R\) so that \(g(x)=f(x)\) and \(g(y)=f(y)\). Each \(z_a\) has the same distances to the two distinct fixed points \(x,y\), so \(g(z_a)=f(z_a)\). Now restrict to \(G[U(xz_a)]\). By induction, all pairwise distances inside this set agree under \(f\) and \(g\). Since \(x,z_a\) are fixed, every vertex of \(U(xz_a)\) is fixed, as a point on the line is determined by its distances to two distinct fixed points. The same argument applies to \(U(yz_a)\). Thus \(f=g\) on \(U(e)\) after a global isometry. Hence \(G[U(e)]\) is globally rigid.
\end{proof}

\begin{corollary}\label{cor:clique-support-globally-rigid}
Let \(Q\) be a clique in \(\operatorname{cl}_3(G)\), and let \(U(Q)\) be the union of \(Q\) and the sets \(U(e)\) over its edges. Then \(G[U(Q)]\) is globally rigid in \(\mathbb R\).
\end{corollary}

\begin{proof}
The case \(|Q|\leq1\) is trivial, and the case \(|Q|=2\) follows from Lemma \ref{lem:support-globally-rigid}. Assume \(|Q|\geq3\). Let \(f,g\colon U(Q)\to\mathbb R\) be injective realizations of \(G[U(Q)]\) with the same edge lengths. For each \(xy\in\binom{Q}{2}\), the restrictions of \(f\) and \(g\) to \(G[U(xy)]\) have the same edge lengths. By Lemma \ref{lem:support-globally-rigid}, the distance \(xy\) is therefore the same under \(f\) and \(g\). Hence the full distance matrix on \(Q\) agrees. After a global isometry, fix two vertices of \(Q\); every other vertex of \(Q\) is then fixed by its distances to them. Finally, for each \(xy\in\binom{Q}{2}\), the vertices \(x,y\) are fixed, and the induced graph \(G[U(xy)]\) is globally rigid. Hence every vertex of \(U(xy)\) is fixed. Thus all vertices of \(U(Q)\) are fixed.
\end{proof}

\begin{proof}[Proof of Theorem \ref{THM : Worst-Case-Bound-Line} assuming Theorem \ref{Theorem : Graphs with bounded intersection}]
Put \(H=\operatorname{cl}_3(G)\). By Proposition \ref{prop:line-closure-distances}, every edge of \(H\) is reconstructable from \((G,d_{G,f})\), for every injective \(f\colon[n]\to\mathbb R\). Moreover, \(H\) is \(3\)-closed, so every non-edge of \(H\) has at most two common neighbours. Since \(G\subseteq H\), we have \(|E(H)|\geq|\mathcal P|\geq Cn^{3/2}\). Taking \(C\) large enough, Theorem \ref{Theorem : Graphs with bounded intersection} applies to \(H\) with \(k=2\), giving a clique \(Q\subseteq[n]\) of size \(|Q|=\Omega(|E(H)|/n)\geq\Omega(|\mathcal{P}|/n)\). By Corollary \ref{cor:clique-support-globally-rigid}, the induced subgraph \(G[U(Q)]\) is globally rigid in \(\mathbb R\), and \(|U(Q)|\geq |Q|=\Omega(|\mathcal P|/n)\).
\end{proof}

\subsection{Proof of Theorem \ref{Theorem : Graphs with bounded intersection}}

Let $G=([n],E)$ be a graph satisfying the assumptions of Theorem \ref{Theorem : Graphs with bounded intersection}. We use the following standard lemma:

\begin{lemma}\label{Lemma : Pruning Graph}
Let $G$ be a graph of average degree $d=\frac{2|E|}{n}$. There exists an induced subgraph $G'$ with \emph{minimum} degree $\delta(G')$ satisfying $ \delta(G') \geq \frac{d}{2}$.
\end{lemma}

\begin{proof}
    Iteratively delete vertices of degree smaller than $d/2$. The process cannot delete every vertex: if it did, each edge would be counted once, at the deletion of its first endpoint, while fewer than $d/2$ edges are deleted at each step. Thus the total number of deleted edges would be less than $nd/2=|E|$, a contradiction. The graph induced on the undeleted vertices satisfies the lemma.
\end{proof}
\noindent We also use the following lemma of Corrádi, whose proof appears in \cite[Lemma~2.1]{Jukna}:

\begin{lemma}\label{LEMMA : Corradi's Lemma}
If $S_1,\ldots,S_N$ are subsets of a finite set $X$, each of size at least $r$, and $|S_i \cap S_j| \leq k$ for all $i\neq j$, then
\[
    |X| \geq \frac{r^2 N}{r+(N-1)k}.
\]
\end{lemma}
\noindent We now bound the size of a maximum independent set in $G$ by applying Corrádi's Lemma:

\begin{claim} \label{CLAIM : Independent Set Bound}
    Let $k\geq1$, and let $G$ be a graph on $n$ vertices such that every two distinct non-adjacent vertices $i,j$ satisfy $|\Gamma(i) \cap \Gamma(j)| \leq k$. If the minimum degree of $G$ satisfies $\delta \geq \sqrt{2nk}$ then the size of a maximum independent set in $G$ satisfies $\alpha(G) \leq 2 \frac{n}{\delta }$.
\end{claim}

\begin{proof}
    Denote the elements of the maximum independent set by $s_1,\ldots,s_{\alpha}$ with $\alpha:=\alpha(G)$. By Lemma \ref{LEMMA : Corradi's Lemma}, applied to $S_i = \Gamma(s_i)$, we have

    \begin{align*}
        n \geq \frac{\delta^2 \alpha}{\delta+k(\alpha-1)}.
    \end{align*}

    Rewriting this as an upper bound on $\alpha(G)$, and using our assumption $\delta \geq \sqrt{2nk}$, which implies $\delta^2-kn\geq \frac12\delta^2$, we have:

    \begin{align*}
        \alpha=\alpha(G)\leq \frac{n (\delta -k) }{\delta^2-kn} \leq \frac{n \delta}{\delta^2/2} = 2\frac{n}{\delta}.
    \end{align*}
\end{proof}

\begin{proof}[Proof of Theorem \ref{Theorem : Graphs with bounded intersection}]
    Let $e=|E|$. By Lemma \ref{Lemma : Pruning Graph}, there is an induced subgraph $G'=(W,F)$ with minimum degree $\delta\geq \frac{e}{n}$. Put $N=|W|$. The common-neighbour condition is inherited by induced subgraphs, so every non-edge in $G'$ has at most $k$ common neighbours.

    Since $N\leq n$ and $\delta\geq e/n\geq 8\sqrt{kn}$, Claim \ref{CLAIM : Independent Set Bound} applies to $G'$. Let $I=\{s_1,\ldots, s_{\alpha}\}$ be any independent set of maximum size in $G'$, so $\alpha \leq 2 \frac{N}{\delta }$. Define
    \[
        B_i=\{u\in W:\Gamma_{G'}(u)\cap I=\{s_i\}\}.
    \]
    The following observation is the heart of the argument:

    \begin{observation}
        Each $B_i$ forms a clique in $G'$.
    \end{observation}

    The observation follows since, if $x,y\in B_i$ are non-adjacent, then $x$ and $y$ have no neighbours in $I\setminus\{s_i\}$. Therefore $(I\setminus\{s_i\})\cup\{x,y\}$ is an independent set larger than $I$, a contradiction.

    The sets $B_i$ are pairwise disjoint. Since each $B_i$ is a clique, one of them has size at least the average size. Thus:
    \begin{align}
        \omega(G') \geq \frac{|\bigcup_{i=1}^{\alpha} B_i|}{\alpha}. \label{EQUATION : Bounding Cliques}
    \end{align}

    To estimate $|\bigcup_{i=1}^{\alpha} B_i|$, write $A_i=\Gamma_{G'}(s_i)$. The set $\bigcup_i B_i$ is precisely the set of vertices lying in exactly one of the sets $A_i$. Count vertices according to how many of the sets $A_i$ contain them. A vertex contained exactly once contributes $1$ to the right-hand side below; a vertex contained twice contributes $0$; and a vertex contained at least three times contributes non-positively. Hence

    \begin{align*}
        \bigg| \bigcup_{i=1}^{\alpha} B_i \bigg|
        \geq
        \sum_i |A_i|-2\sum_{i<j}|A_i\cap A_j| .
    \end{align*}

    Since the vertices $s_i$ are pairwise non-adjacent, every intersection $A_i\cap A_j$ has size at most $k$. Hence

    \begin{align*}
        \bigg| \bigcup_{i=1}^{\alpha} B_i \bigg|
        \geq
        \delta\alpha-2k{\alpha \choose 2}.
    \end{align*}

    \noindent Plugging this into Equation \ref{EQUATION : Bounding Cliques} and using Claim \ref{CLAIM : Independent Set Bound} gives

    \begin{align*}
        \omega(G)\geq\omega(G') \geq \delta-k(\alpha-1) \geq \delta-k\alpha \geq \delta - 2k\frac{N}{\delta}.
    \end{align*}
    Since $\delta^2\geq64kn\geq4kN$, the last expression is at least $\delta/2$. As $\delta\geq e/n$, this gives $\omega(G)\geq\frac{\delta}{2}\geq\frac{e}{2n}\geq\frac{e}{4n}$.
\end{proof}

The clique size guaranteed by Theorem \ref{Theorem : Graphs with bounded intersection} is asymptotically tight. Indeed, take $\frac{n}{q}$ disjoint copies of $K_q$. Then $|E|=\Theta(nq)$, while $\omega(G)=q=\Theta(|E|/n)$. Moreover, non-edges lie in different components and therefore have no common neighbours.
Our edge density assumption is also tight:
\begin{claim}\label{Claim : Projective Plane Graph}
    For every integer $k\geq1$ and infinitely many $N$, there is a graph $G$ on $N$ vertices such that every non-edge has at most $k$ common neighbours, $|E(G)|=\Theta\left(\sqrt{k}N^{3/2}\right)$, and $\omega(G)=O(k)$.
\end{claim}

\begin{proof}
    Let $F$ be the bipartite incidence graph between points and affine lines in $\mathbb{F}_p^2$. Then $F$ has $m=\Theta(p^2)$ vertices, $\Theta(p^3)=\Theta(m^{3/2})$ edges, and is $C_4$-free.

    Replace each vertex of $F$ by a $k$-clique and each edge of $F$ by all possible edges between the corresponding cliques. Let $G$ be the resulting graph and put $N=km$. If two vertices of $G$ are non-adjacent, then they lie in clusters corresponding to two non-adjacent vertices of $F$. Since $F$ is $C_4$-free, those two vertices have at most one common neighbour in $F$, and hence the two blown-up vertices have at most $k$ common neighbours in $G$. Also,

    \[
        |E(G)|=\Theta(k^2m^{3/2})=\Theta(\sqrt{k}N^{3/2}).
    \]

    Since $F$ is bipartite, every clique of $G$ uses vertices from at most two clusters, and two clusters form a clique exactly when the corresponding vertices of $F$ are adjacent. Thus $\omega(G)=2k$.
\end{proof}

\paragraph{Reconstructing $3$ points on the $3$-regular tree requires $\Omega(n^2)$ distances}

It is natural to ask whether we can find analogs of Theorem \ref{THM : Worst-Case-Bound-Line} for subsets $S$ of metric spaces that are not $\R$ or $S^1$. We now consider the 3-regular tree $\mathbb{T}_3$ with the graph metric. We give an example with quadratically many known distances for which no subset of size 3 is reconstructable.

Consider the infinite binary tree, and associate each vertex of distance $\ell$ from some root $r$ with a length $\ell$ string over $\{0,1\}$. Fix $m$, and let
\[
    S_1=\{x\in\{0,1\}^{3m}: \text{$x$ starts with $2m$ zeros}\},
    \qquad
    S_2=\{x\in\{0,1\}^{3m}: \text{$x$ ends with $2m$ zeros}\}.
\]
Then $|S_1|=|S_2|=2^m=:t$. Every two vertices of $S_1$ are at distance at most $2m$, and every two distinct vertices of $S_2$ are at distance at least $4m$.

In the $3$-regular tree, for an edge $e$ the graph $\mathbb{T}_3 \setminus{e}$ is isomorphic to two copies of the infinite binary tree. Let $R_i$ be the configuration obtained by taking one copy of $S_i$ in each component of $\mathbb{T}_3 \setminus{e}$. Label both configurations by the side of $e$ and by an element of $\{0,1\}^m$. If $\mathcal{P}$ is the set of all pairs with endpoints in different components, then $|\mathcal{P}|=t^2$ and all distances in $\mathcal{P}$ are $6m+1$, for both $R_1$ and $R_2$.

Every triple of points contains two points in the same component of $\mathbb{T}_3\setminus e$. In $R_1$, the distance between such a pair is at most $2m$, whereas in $R_2$ it is at least $4m$. Hence no triple is reconstructable. In this limited sense the example is sharp: the known-distance graph is $K_{t,t}$, which has $t^2$ edges and is triangle-free. By Mantel's theorem, any graph on $2t$ vertices with more than $t^2$ edges contains a triangle; such a triangle is reconstructable since all its distances are known.

\subsection{Sparse globally rigid graphs in $\mathbb{R}$}

We finish the section with a construction and a lower bound for sparse
globally rigid graphs on the line. Given a graph \(G\), let \(T(G)\) be the
graph obtained by deleting every edge \(e=uv\) and replacing it by three
internally disjoint \(u\)-\(v\) paths of length two. Thus \(T(G)\) has
\(|V(G)|+3|E(G)|\) vertices and \(6|E(G)|\) edges.

\begin{lemma}
    If $G$ is globally rigid in $\mathbb{R}$, then $T(G)$ is globally rigid in $\mathbb{R}$.
\end{lemma}

\begin{proof}
    Let \(f,g\) be two injective realizations of \(T(G)\) with the same edge
    lengths. If \(G\) has at most one vertex there is nothing to prove.
    Otherwise \(G\) has an edge, since a globally rigid graph with at least
    two vertices is connected. For every old edge \(uv\in E(G)\), the three
    new subdivision vertices are common neighbours of \(u\) and \(v\) in
    \(T(G)\). By Lemma \ref{lem:three-common-neighbours-line},
    \(\dist(f(u),f(v))=\dist(g(u),g(v))\). Hence the restrictions of \(f\)
    and \(g\) to the old vertices realize the same edge lengths of \(G\).
    Since \(G\) is globally rigid, \emph{all} distances between old vertices \emph{must} agree.

    Choose an old edge \(uv\). After composing \(g\) with an isometry of
    \(\mathbb R\), we may assume \(g(u)=f(u)\) and \(g(v)=f(v)\). Every old
    vertex is then fixed, since its distances to the two distinct points
    \(u,v\) agree under \(f\) and \(g\). Finally, each new subdivision vertex
    is adjacent to the two old endpoints of the edge it subdivides, so it is
    also fixed by its distances to two distinct fixed points. Thus \(f\) and
    \(g\) differ by a global isometry, and \(T(G)\) is globally rigid.
\end{proof}

Starting with $G_1=K_3$ and setting $G_{i+1}=T(G_i)$ gives an infinite family of globally rigid graphs. If $a_i=|E(G_i)|/n_i$, where $n_i$ is the number of vertices of $G_i$, then $a_{i+1}=6a_i/(1+3a_i)$ and $a_1=1$, so $a_i\to 5/3$. Hence the average degree tends to $10/3$.
We also claim the following lower bound on the average degree:

\begin{proposition}
    Any globally rigid graph with at least $4$ vertices has average degree of at least $\frac{12}{5}$.
\end{proposition}

\begin{proof}
    Let $G=([n],E)$ be globally rigid. First, \(G\) is connected: otherwise
    one can translate one connected component in an injective realization,
    preserving all edge lengths while changing a distance to a vertex outside
    the component. Also \(G\) has minimum degree at least \(2\). Indeed,
    connectedness rules out isolated vertices. If \(x\) is a leaf with
    neighbour \(y\), choose \(u\in[n]\setminus\{x,y\}\) and a generic
    injective realization \(f\) such that reflecting \(f(x)\) about \(f(y)\)
    does not collide with another vertex and changes the distance from \(x\)
    to \(u\). Replacing \(f(x)\) by \(2f(y)-f(x)\) preserves every edge
    length, since \(x\) is a leaf, contradicting global rigidity.

    Let $D_2$ be the set of vertices of degree $2$. We claim that $D_2$ is
    independent. Suppose that \(x,y\in D_2\) and \(xy\in E\). Let \(w\) be
    the other neighbour of \(x\), and let \(z\) be the other neighbour of
    \(y\). Fix an injective map
    \(\varphi\colon[n]\setminus\{x,y\}\to\mathbb R\).

    If \(w\neq z\), compare the two maps obtained by fixing all other
    vertices and placing \(x,y\) at
    \(\varphi(w)+c,\varphi(z)+c\) in one map and at
    \(\varphi(w)-c,\varphi(z)-c\) in the other. The only edge lengths that
    can change are \(xw,yz,xy\), and in the two maps they are
    \(|c|,|c|\), and \(|\varphi(w)-\varphi(z)|\). For all but finitely many
    choices of \(c\), the two maps are injective realizations with the same
    edge lengths, but the distance from \(x\) to \(z\) changes. This
    contradicts global rigidity.

    It remains to consider \(w=z\). Since \(n\ge4\), choose
    \(u\in[n]\setminus\{w,x,y\}\). Compare the two maps obtained by fixing
    all other vertices and placing \(x,y\) at
    \(\varphi(w)+c,\varphi(w)+c+a\) in one map and at
    \(\varphi(w)-c,\varphi(w)-c-a\) in the other. The only edge lengths that
    can change are \(wx,wy,xy\), and in both maps they are
    \(|c|,|c+a|\), and \(|a|\). For a generic choice of \(a,c\), the two
    maps are injective realizations with the same edge lengths, but the
    distance from \(u\) to \(x\) changes. This again contradicts global
    rigidity. Hence \(D_2\) is independent.

    If $|D_2| \geq 3n/5$, then the edges incident to vertices of $D_2$ are all distinct, so $|E|\geq 2|D_2|\geq 6n/5$. Otherwise,
    \[
        \sum_{v\in [n]}\deg(v)\geq 2|D_2|+3(n-|D_2|)\geq 12n/5.
    \]
    In both cases the average degree is at least $12/5$.
\end{proof}

\section{The Random Setting}

Throughout this section we fix a labelled configuration
\(V=(v_1,\ldots,v_n)\) of distinct points in \(\R^d\). Equivalently, one may
think of an injective map \(f\colon[n]\to\R^d\), but we write
\(v_i=f(i)\). All isometries below are Euclidean isometries of \(\R^d\), and
throughout this section \(n\ge2\).

There is a simple obstruction to \emph{complete} reconstruction. In
\(\R^2\), if \(n-2\) points lie on a line \(L\), and two points \(x,y\) lie
outside it, then reflecting \(y\) across \(L\) gives a configuration with all
distances except the one between \(x\) and \(y\) the same. In general, a
natural obstruction in \(\R^d\) is the existence of affine subspaces
containing almost all of the points. The aim of this section is twofold:
first, to show that this is essentially the only obstruction to complete
reconstruction; second, to show that for arbitrary \(V\), most points can be
reconstructed from randomly revealed distances.

A labelled \(d\)-dimensional configuration is an ordered tuple
\(V=(v_1,\ldots,v_n)\) of distinct points in \(\R^d\). Given another
configuration \(U=(u_1,\ldots,u_n)\), let
\(B(U,V)\) be the graph on \([n]\) whose edges are the pairs \(\{i,j\}\) for
which \(\dist(u_i,u_j)\ne\dist(v_i,v_j)\). We set
\[
    \dist_{\PD}(U,V):=\frac{|E(B(U,V))|}{\binom n2}.
\]
Here \(\PD\) stands for pairwise-distance disagreement.

If \(G\) is a graph on \([n]\), we say that a labelled configuration
\(U=(u_1,\ldots,u_n)\) is \(G\)-compatible with \(V\) if
\(\dist(u_i,u_j)=\dist(v_i,v_j)\) for every \(ij\in E(G)\).
Thus, to reconstruct \(V\), it is enough to show that \(G\) is not compatible
with any configuration \(U\) not isometric to \(V\).

\begin{definition}[LCI distance of configurations]
For two labelled configurations \(U=(u_1,\ldots,u_n)\) and
\(V=(v_1,\ldots,v_n)\), define
\[
        \dist_{\LCI}(U,V)
        :=\min_T \frac{1}{n}\bigl|\{i\in[n]:T u_i\ne v_i\}\bigr|,
\]
where the minimum is over all Euclidean isometries \(T\) of \(\R^d\). Here
\(\LCI\) stands for largest common isometric part.
\end{definition}

\noindent We recall the following well-known fact from Euclidean geometry.

\begin{fact}\label{fact:reference-points-d}
Let $0\le r\le d$, and let $a_0,\ldots,a_r\in\R^d$ be affinely independent.
If $x\in \aff(a_0,\ldots,a_r)$, then $x$ is uniquely determined among all
points of $\R^d$ by
\[
        \dist(x,a_0),\ldots,\dist(x,a_r).
\]
In particular, if $r=d$, then every $x\in\R^d$ is uniquely determined by its
distances to $d+1$ affinely independent reference points.
\end{fact}

The next lemma shows that if two configurations have small PD distance, we can find an isometry matching most of their points.

\begin{lemma}
\label{lem:dense-agreement-d}
For every $d\ge1$ and every two labelled configurations
\(U=(u_1,\ldots,u_n)\) and \(V=(v_1,\ldots,v_n)\),
\[
        \dist_{\LCI}(U,V)
        \le (d+2)\sqrt{\dist_{\PD}(U,V)} .
\]
\end{lemma}

\begin{proof}
Set $B=B(U,V)$, $p=\dist_{\PD}(U,V)$, and $\rho=\sqrt p$. If
$p\ge (d+2)^{-2}$, the claim is trivial. Assume $p<(d+2)^{-2}$, and let
\(S=\{i\in[n]:\deg_B(i)\le \rho n\}\). Observe that
\(|[n]\setminus S|\le \rho n\): if \(\rho=0\) this is clear. For
\(\rho>0\), set \(W=[n]\setminus S\). Then
\(|W|\rho n<\sum_{i\in W}\deg_B(i)\le2|E(B)|=pn(n-1)\), hence
\(|W|<\rho(n-1)\). In particular \(S\neq\emptyset\).

Choose an inclusion-maximal set $A\subseteq S$ with the following two
properties: the points $\{v_a:a\in A\}$ are affinely independent, and no two
vertices of $A$ are adjacent in $B$. Write $q=|A|$, so \(1\le q\le d+1\). Since $B[A]$
has no edges, all pairwise distances inside $A$ agree for $U$ and $V$. As
$\{v_a:a\in A\}$ is affinely independent, the same is true for
$\{u_a:a\in A\}$; hence the congruence \(u_a\mapsto v_a\) on \(A\) extends
to a Euclidean isometry \(T\) of \(\R^d\). Thus \(T u_a=v_a\) for every
\(a\in A\).

Let \(F=\aff\{v_a:a\in A\}\).
We claim that \(T\) matches every \(i\in S\) that either lies in \(A\), or
has no \(B\)-edge to any vertex of \(A\). This is clear for \(i\in A\). Now
let \(i\in S\setminus A\) have no \(B\)-edge to \(A\). If \(v_i\notin F\),
then \(A\cup\{i\}\) still has
affinely independent \(V\)-points and spans no edge of $B$, contradicting
the maximality of $A$. Hence $v_i\in F$. Moreover, since $\{a,i\}\notin E(B)$
for every $a\in A$, the points $T u_i$ and $v_i$ have the same distances
to all reference points $\{v_a:a\in A\}$. By
Fact \ref{fact:reference-points-d}, applied with \(r=|A|-1\), this implies
\(T u_i=v_i\).

Thus the only labels in \(S\) which may fail to be matched by \(T\) are those
whose distance to some label in \(A\) disagrees between \(U\) and \(V\). Since
every vertex of \(A\) lies in \(S\), there are at most
\(\sum_{a\in A}\deg_B(a)\le q\rho n\le(d+1)\rho n\) such labels. Thus \(T\)
mismatches at most \((d+1)\rho n\) labels in \(S\), plus possibly the
\(\rho n\) labels of \([n]\setminus S\). Hence
\(\dist_{\LCI}(U,V)\le(d+2)\rho=(d+2)\sqrt{\dist_{\PD}(U,V)}\).
\end{proof}

We now claim that, in general, distinguishing \(V\) from configurations
\(U\) of large \(\PD\)-distance can be done by randomly sampling a few pairs.
Thus the main difficulty is handling configurations with small
\(\PD\)-distance.

\begin{lemma}
\label{lem:sampling-separates-PD}
Fix $d\ge 1$ and $a\in(0,1)$. There is a constant \(c_{d,a}=O_a(d)\) such that the
following holds for sufficiently large \(n\). Let
\(V=(v_1,\ldots,v_n)\) be a labelled configuration of points in $\mathbb{R}^d$, and let
\(E'\subseteq \binom{[n]}2\) be obtained by keeping each pair independently
with probability \(q=c_{d,a}/n\).

Then, with probability \(1-o(1)\),
simultaneously for all configurations \(U=(u_1,\ldots,u_n)\) with
\(\dist_{\PD}(U,V)\ge a\) we have \(E'\cap E(B(U,V))\neq\emptyset\).
\end{lemma}

\begin{proof}
It is enough to prove the lemma for \(0<a<1/2\); for \(a\ge1/2\), take
\(c_{d,a}=c_{d,1/2}\).
Put \(\Omega=\binom{[n]}2\), and set
\(\mathcal B_V=\{E(B(U,V)):U\in(\R^d)^n\}\subseteq2^\Omega\) and
\(\mathcal B_{V,a}=\{B\in\mathcal B_V:|B|\ge a|\Omega|\}\). It is enough to
show that \(\operatorname{VCdim}(\mathcal B_V)\le C_0dn\) for an absolute
constant \(C_0\). Indeed, by the \(\varepsilon\)-net theorem
\cite[Theorem~6.10]{zbMATH07478425}, with \(\varepsilon=a\) and
\(\delta=e^{-n}\), a \(q\)-random subset of \(\Omega\) hits every member of
\(\mathcal B_{V,a}\) provided
\[
        q|\Omega|
        \ge
        C_{\mathrm{net}}a^{-1}\bigl(C_0dn\log(1/a)+n\bigr).
\]
Thus \(c_{d,a}=4C_{\mathrm{net}}a^{-1}(C_0d\log(1/a)+1)\) works for all large
\(n\), and in particular \(c_{d,a}=O_a(d)\).

It remains to bound the VC dimension of \(\mathcal{B}_V\). Let \(Y\subseteq\Omega\), \(|Y|=m\).
For \(ij\in Y\), write \(U=(u_1,\ldots,u_n)\),
\(u_i=(x_{i,1},\ldots,x_{i,d})\), and set
\(F_{ij}(U)=\sum_{\ell=1}^d (x_{i,\ell}-x_{j,\ell})^2-\|v_i-v_j\|^2\). The
agreement edges projected to \(Y\) are exactly the zero set of the family
\(\{F_{ij}:ij\in Y\}\). Thus the number of possible projections to \(Y\) is
at most the number of zero/nonzero patterns realized by these \(m\) quadratic
polynomials in \(dn\) variables. If \(Y\) is shattered, all \(2^m\)
agreement/disagreement patterns occur. When \(m<dn\), this already gives
\(m=O(dn)\). Otherwise these patterns are zero/nonzero patterns of \(m\)
polynomials of degree \(2\) in \(dn\) variables. By Warren's theorem
\cite[Theorem~10.4]{zbMATH06484571}, at most
\(\left(Cm/(dn)\right)^{dn}\) such patterns occur, for some universal
constant \(C>0\). Hence
\[
        2^m\le \left(C\frac{m}{dn}\right)^{dn}.
\]
Writing \(t=m/(dn)\), this implies \(2^t\le Ct\), and therefore \(t=O(1)\).
Thus \(m=O(dn)\), as required. All bounds above are independent of the
particular configuration \(V\).
\end{proof}

Applying the same argument, we get:

\begin{corollary}\label{cor:variable-scale-PD-separation}
Fix \(d\ge1\), let \(D=pn\to\infty\), and put
\(\alpha=D^{-1/3}\) and \(\eta=\alpha^2/(d+2)^2\). Let
\(V=(v_1,\ldots,v_n)\) be a labelled configuration in \(\mathbb R^d\), and
let \(G\sim G(n,p)\). Then, with probability \(1-o(1)\), every
\(G\)-compatible configuration \(U\) satisfies
\(\dist_{\PD}(U,V)<\eta\), and hence \(\dist_{\LCI}(U,V)<\alpha\). The
\(o(1)\) is uniform over \(V\).
\end{corollary}

\begin{proof}
Use the proof of Lemma \ref{lem:sampling-separates-PD} with
\(\varepsilon=\eta\). The family of disagreement graphs has VC-dimension
\(O(dn)\), so the VC \(\varepsilon\)-net theorem requires
\(O_d(\eta^{-1}n\log(1/\eta))=O_d(nD^{2/3}\log D)\) sampled pairs. Since
\(p\binom n2=(1+o(1))Dn/2\) and \(D^{1/3}/\log D\to\infty\), the random
graph hits every disagreement graph of size at least \(\eta\binom n2\) w.h.p.
Thus w.h.p. every compatible \(U\) has \(\dist_{\PD}(U,V)<\eta\). The final
conclusion follows from Lemma \ref{lem:dense-agreement-d}.
\end{proof}

\subsection{Complete reconstruction of spanning point sets}

Call a configuration \(V=(v_1,\ldots,v_n)\) \(\tau\)-spanning if, for every
\(I\subseteq[n]\) with \(|I|>\tau n\), one has
\(\aff\{v_i:i\in I\}=\R^d\). For a graph \(G\) with vertex set \([n]\), write
\(N_G(i)\) for the neighbours of \(i\).

\begin{lemma}\label{lem:no-small-anchor-closed-sets}
Fix \(d\ge1\), \(0<\tau<1\), \(C>0\), and \(0<\gamma<1-\tau\). Put
\(\theta=1-\tau-\gamma\), and assume \(C\theta>1\). Let
\(V=(v_1,\ldots,v_n)\) be \(\tau\)-spanning, and let \(G\sim G(n,p)\) with
\(p=C\ln n/n\). Then, w.h.p., for every nonempty \(M\subseteq[n]\) with
\(|M|\le\gamma n\), there is \(i\in M\) such that \(N_G(i)\setminus M\)
contains \(d+1\) labels whose \(V\)-points are affinely independent. The
\(o(1)\) is uniform for all \(\tau\)-spanning configurations \(V\).
\end{lemma}

\begin{proof}
Consider \(W\subseteq[n]\) with \(|W|\ge(1-\gamma)n\), and let \(R\) be a
\(p\)-random subset of \(W\). We claim that the probability that \(R\) contains
no \(d+1\) affinely independent \(V\)-points is at most \(n^{-C\theta+o(1)}\).
If \(R\) has no \(d+1\) affinely independent labels, take a maximal
affinely independent \(A\subseteq R\). Then \(|A|\le d\), and
\(R\subseteq L(A):=\{j:v_j\in\aff\{v_a:a\in A\}\}\), with
\(L(\emptyset)=\emptyset\). Since \(V\) is \(\tau\)-spanning,
\(|L(A)|\le\tau n\), and hence \(|W\setminus L(A)|\ge\theta n\). For fixed
\(A\), the probability that \(A\subseteq R\subseteq L(A)\) is at most
\(p^{|A|}e^{-\theta pn}\). Summing over all \(|A|\le d\) gives
\[
    \PP(R\text{ contains no }d+1\text{ affinely independent }V\text{-points})
    \le e^{-\theta pn}\sum_{r=0}^{d}\binom{n}{r}p^r
    \le n^{-C\theta+o(1)},
\]
as \(pn=C\ln n\).

Now fix \(M\subseteq[n]\), \(|M|=k\le\gamma n\), and put
\(W=[n]\setminus M\). The sets \(N_G(i)\setminus M\), \(i\in M\), are
independent \(p\)-random subsets of \(W\). Choose \(\eta>0\) such that
\(C\theta>1+2\eta\). By the estimate above, for large enough \(n\), each one of
these sets fails to contain \(d+1\) affinely independent \(V\)-points with
probability at most \(n^{-1-\eta}\). Hence
\(\PP(M\text{ violates the desired conclusion})\le n^{-k(1+\eta)}\). A union bound
over \(1\le k\le\gamma n\) gives
\[
\sum_{k=1}^{\lfloor\gamma n\rfloor}
        \binom{n}{k} n^{-k(1+\eta)}
    \le
\sum_{k\ge1}
        \left(\frac{e n^{-\eta}}{k}\right)^k
    =o(1),
\]
which proves the lemma.
\end{proof}

\begin{theorem}\label{thm:random-reconstruction-tau-spanning}
Fix \(d\ge1\), \(0<\tau<1\), and \(C>1/(1-\tau)\). Let
\(V=(v_1,\ldots,v_n)\) be \(\tau\)-spanning, and let \(G\sim G(n,p)\), where
\(p=C\ln n/n\). Then, w.h.p., for every \(G\)-compatible configuration
\(U=(u_1,\ldots,u_n)\), there is an isometry \(T\) such that
\(T u_i=v_i\) for every \(i\in[n]\). Equivalently, the revealed distances
determine \(V\) up to isometry. The \(o(1)\) is uniform for all
\(\tau\)-spanning configurations \(V\).
\end{theorem}

\begin{proof}
Choose \(0<\gamma<1-\tau\) so small that
\(C(1-\tau-\gamma)>1\), and put \(a=(\gamma/(d+2))^2\). Also assume that
\(n\) is large enough so \(p\ge c_{d,a}/n\). By
Lemma \ref{lem:sampling-separates-PD} and Lemma \ref{lem:no-small-anchor-closed-sets},
w.h.p. the following two events hold. First, every \(U\) with
\(\dist_{\PD}(U,V)\ge a\) is not \(G\)-compatible with \(V\). Second, every
nonempty \(M\subseteq[n]\) with \(|M|\le\gamma n\) contains an index
\(i\in M\) such that \(N_G(i)\setminus M\) contains \(d+1\) labels whose
\(V\)-points are affinely independent. We show that, on these events, \(V\)
is reconstructable.

Let \(U=(u_1,\ldots,u_n)\) be \(G\)-compatible with \(V\). By the first
event, \(\dist_{\PD}(U,V)<a\). Applying Lemma \ref{lem:dense-agreement-d}, there
is an isometry \(T\) such that, for \(M:=\{i\in[n]:T u_i\ne v_i\}\), we have
\(|M|\le\gamma n\). We claim that \(M=\emptyset\), which proves the theorem.
Indeed, suppose not. By the second event, there is \(i\in M\) such that
\(N_G(i)\setminus M\) contains labels \(a_0,\ldots,a_d\) whose \(V\)-points
are affinely independent. For every \(r=0,\ldots,d\), we have
\(\dist(T u_i,v_{a_r})=\dist(u_i,u_{a_r})=\dist(v_i,v_{a_r})\), where the
first equality uses \(T u_{a_r}=v_{a_r}\), and the second uses
\(G\)-compatibility of \(U\). By Fact \ref{fact:reference-points-d}, applied to
the reference points \(v_{a_0},\ldots,v_{a_d}\), this forces \(T u_i=v_i\),
contradicting \(i\in M\). Hence \(M=\emptyset\), so \(U\) is isometric to
\(V\).
\end{proof}

We remark that one can't replace the assumption $C>1/(1-\tau)$ by a smaller constant. Indeed, take a collection of $\tau n$ points of $V$ spanning a $d-1$ flat $F$, and let $A$ denote their corresponding label set, the rest of the points are arbitrary outside $F$.  Then $G[[n]-A]$ is distributed as $G(n',p)$ with $n'>(1-\tau)n$ and $p=C \log(n)/n <(1-\varepsilon)\log(n')/n'$ for some $0<\varepsilon<1$, and $n'=n-|A|\geq (1-\tau)n$. Hence w.h.p there two isolated vertices in $G[[n]-A]$. The distance between these points can't be determined from distances along $G$, in similar fashion to the two dimensional case.

\subsection{Reconstructing a giant component}

\begin{proposition}\label{prop:partial-reconstruction-upper}
Fix \(d\ge1\), let \(D=pn\to\infty\), and let \(G\sim G(n,p)\). For every
labelled configuration \(V=(v_1,\ldots,v_n)\) of distinct points in
\(\mathbb R^d\), w.h.p. there is a reconstructable set \(A\subseteq[n]\) with
\(|A|=n-o(n)\). The failure probability and the \(o(n)\) term are uniform
over \(V\).
\end{proposition}

\begin{proof}
We prove the statement by induction on
\(s=\dim\aff\{v_1,\ldots,v_n\}\). The case \(s=0\) is trivial. Assume
\(s\ge1\), put \(F=\aff\{v_i:i\in[n]\}\), and set
\(\alpha=D^{-1/3}\), \(\rho=4\alpha\).

First suppose that no affine \((s-1)\)-flat in \(F\) contains more than
\((1-\rho)n\) labels. Call a set \(M\subseteq[n]\) bad if, for every
\(i\in M\), the set \(N_G(i)\setminus M\) contains no \(s+1\) labels whose
\(V\)-points are affinely independent.

We claim that w.h.p. there is no bad set \(M\) with
\(\alpha n\le |M|\le2\alpha n\). Fix such \(M\), write \(m=|M|\), and put
\(W=[n]\setminus M\). If \(R\subseteq W\) is \(p\)-random, then the probability
that \(R\) contains no \(s+1\) affinely independent labels is at most
\(\sum_{r=0}^{s}\binom nr p^r e^{-2\alpha D}\le (1+D)^s e^{-2\alpha D}\le
e^{-\alpha D}\). Indeed, after fixing a maximal affinely independent
\(A\subseteq R\), all of \(R\) must lie in
\(L_A:=\aff\{v_a:a\in A\}\), with \(L_\emptyset=\emptyset\). If
\(A\ne\emptyset\), then \(L_A\) is contained in an affine \((s-1)\)-flat in
\(F\); hence in all cases \(L_A\) contains at most \((1-\rho)n\) labels. Thus
\(|W\setminus L_A|\ge \rho n-|M|\ge(\rho-2\alpha)n=2\alpha n\).
Independence over \(i\in M\) gives probability at most \(e^{-\alpha Dm}\).
Therefore
\[
        \sum_{m=\lceil\alpha n\rceil}^{\lfloor2\alpha n\rfloor}
        \binom nm e^{-\alpha Dm}
        \le
        \sum_m\left(\frac e\alpha e^{-\alpha D}\right)^m=o(1),
\]
since \(\alpha D=D^{2/3}\gg\log(1/\alpha)\).

Work on this event and on the event of Corollary
\ref{cor:variable-scale-PD-separation}. For a compatible \(U\) and an
isometry \(T_U\), set \(M_{T_U}(U,V):=\{i:T_Uu_i\ne v_i\}\). If
\(|M_{T_U}(U,V)|\le\alpha n\), then \(M_{T_U}(U,V)\) is bad: otherwise some
\(i\in M_{T_U}(U,V)\) has neighbours \(a_0,\ldots,a_s\notin M_{T_U}(U,V)\) whose
\(V\)-points are affinely independent. Since \(T_Uu_{a_j}=v_{a_j}\) and \(U\)
is compatible with \(V\), we have
\(\|T_Uu_i-v_{a_j}\|=\|u_i-u_{a_j}\|=\|v_i-v_{a_j}\|\) for all \(j\). The
points \(v_{a_0},\ldots,v_{a_s}\) affinely span \(F\), and \(v_i\in F\), so
Fact \ref{fact:reference-points-d}, applied with \(r=s\) to \(x=v_i\) and the
reference points \(v_{a_0},\ldots,v_{a_s}\), forces \(T_Uu_i=v_i\), a
contradiction.

Let \(X\) be the union of all sets \(M_{T_U}(U,V)\), where \(U\) is compatible,
\(T_U\) is an isometry, and \(|M_{T_U}(U,V)|\le\alpha n\). A union of bad
sets is bad. If \(|X|>\alpha n\), choose a minimal sub-union whose size is at least
\(\alpha n\). Since each added set has size at most \(\alpha n\),
this finite sub-union has size between \(\alpha n\) and \(2\alpha n\), a
contradiction. Hence \(|X|\le\alpha n\). Let \(A=[n]\setminus X\). For every
compatible \(U\), Corollary \ref{cor:variable-scale-PD-separation} gives an isometry
\(T_U\) with \(|M_{T_U}(U,V)|\le\alpha n\); hence
\(M_{T_U}(U,V)\subseteq X\), so \(T_Uu_i=v_i\) for every \(i\in A\).
Therefore \(A\) is reconstructable and
\(|A|\ge(1-\alpha)n\).

It remains to handle the case where some affine \((s-1)\)-flat \(L\subseteq F\)
contains \(S=\{i:v_i\in L\}\) with \(|S|\ge(1-\rho)n\). The graph \(G[S]\) is
distributed as \(G(|S|,p)\), and \(p|S|=(1-o(1))D\to\infty\). By the induction
hypothesis applied to the subconfiguration \((v_i)_{i\in S}\), whose affine
span has dimension at most \(s-1\), the distances revealed inside \(S\)
reconstruct a set \(A\subseteq S\) with \(|A|=|S|-o(|S|)=n-o(n)\) w.h.p.
Since every \(G\)-compatible configuration restricts to a
\(G[S]\)-compatible configuration on \(S\), the same \(A\) is reconstructable
from the full revealed data.
\end{proof}

\begin{proposition}\label{prop:partial-reconstruction-lower}
Fix \(d\ge1\), let \(p=o(1/n)\), and let \(G\sim G(n,p)\). Then w.h.p., for
every labelled configuration \(V=(v_1,\ldots,v_n)\) of distinct points in
\(\mathbb R^d\), every reconstructable set has size \(o(n)\).
\end{proposition}

\begin{proof}
For any fixed \(c<1\), eventually \(p<c/n\). By the well-known subcritical
random graph bound, w.h.p. all components of \(G\) have size \(O(\log n)\).
It is clear that any reconstructable set with at least two labels lies in one
component. Hence every reconstructable set has size \(o(n)\).
\end{proof}

\begin{proof}[Proof of Theorem \ref{thm:weak-threshold-linear-reconstruction}]
The supercritical regime is Proposition \ref{prop:partial-reconstruction-upper}, and
the subcritical regime is Proposition \ref{prop:partial-reconstruction-lower}.
\end{proof}

\section{Concluding Remarks and Open Problems}

We remark that the proof of our extremal result is inspired by a result of \cite{Solymosi} about large cliques in induced $C_4$-free graphs. We also  believe it is interesting to construct sparse globally rigid graphs with average degree as small as possible, and our results leave a gap to be resolved in this direction.

In the random setting there are many open questions remaining. What is the threshold degree $D$, depending on $d$, such that a random $D$-regular graph w.h.p. reconstructs a giant component? What is the actual constant $c_d$ such that sampling pairs with probability $p=c/n$ gives no giant reconstructable component for $c<c_d$, but gives one for $c>c_d$?

It is also unclear whether randomness is essential for reconstruction. For example, does there exist a high-degree expander \(G=([n],E)\) such that, after a uniformly random permutation \(\pi\) of \([n]\), the distances along the permuted graph \(G_\pi\) reconstruct a giant component?

Volume rigidity has also been studied recently \cite{bulavka2025volume,lew2025kvolume}. Can our random-setting results be extended to volume rigidity, and if so, what do they give?

\paragraph{Acknowledgments} We thank Shir Amir, Dan Bar, Nicolas Curien, Dániel Garamvölgyi, Bo'az Klartag, Eran Nevo, Asaf Petruschka, and Ran Tessler for useful discussions. The proofs in Section 3 were found during several sessions with ChatGPT 5.5 Pro. Before interacting with ChatGPT the authors were aware of this proof strategy, but were not able to handle the high disagreement case, Lemma 19, which we found particularly elegant. We also acknowledge Codex with its assistance with rewriting the paper. The authors have verified all proof in this paper, and take full responsibility for the final text and arguments. We are grateful for the referee for many insightful comments which helped improve the correctness and  presentation. The first author would like to thank the ISF for their support under Grant 14.13.20.

\bibliographystyle{plainurl}
\bibliography{ref}

\appendix
\section{Monotone Paths in Random Graphs}\label{app:monotone-threshold}

A path \(P=(v_0,\ldots,v_t)\) is monotone if \(v_0< v_1<\cdots <v_t\).

\begin{theorem}\label{thm:monotone-threshold}
Let \(G \sim G(n,p)\) be a random graph on \([n]\), with its natural order. Then:
\begin{itemize}
    \item for every \(0<\varepsilon<1\), if \(p=(1-\varepsilon)\frac{\ln n}{n}\), then w.h.p. there is no monotone path from \(1\) to \(n\);
    \item for every \(\varepsilon>0\), if \(p=(1+\varepsilon)\frac{\ln n}{n}\), then w.h.p. there is a monotone path from \(1\) to \(n\).
\end{itemize}
\end{theorem}

\begin{proof}
Let \(N\) be the number of monotone paths from \(1\) to \(n\). A monotone path
with \(k\) internal vertices is determined by choosing \(k\) vertices from
\(\{2,\ldots,n-1\}\), so
\(\EE N=\sum_{k=0}^{n-2}\binom{n-2}{k}p^{k+1}=p(1+p)^{n-2}\le p e^{pn}\).
For \(p=(1-\varepsilon)\ln n/n\), this is \(o(1)\), and Markov's inequality
concludes the subcritical case.

\noindent It remains to analyze the supercritical case.
Write \(p=c\ln n/n\), where \(c=1+\varepsilon>1\).

\begin{claim}\label{clm:one-sided-growth}
Let \(p=c\ln n/n\), where \(c>1\). Then there is \(\tau>1/2\), depending
on \(c\), such that w.h.p. at least \(n^\tau\) vertices in
\(\{1,\ldots,\lfloor n/2\rfloor\}\) are reachable from \(1\) by monotone paths
contained in this set.
\end{claim}

\begin{proof}[Proof of the claim]
Choose \(1<c_0<c\), then choose \(K\) so large that
\(r:=K\ln(1+c_0/K)>1\); this is possible since \(r\to c_0\). Choose
\(\delta>0\) so small that
\[
    \left(\frac12-2\delta\right)r>\frac12,
\]
and then fix
\[
    \frac12<\tau<\min\left\{1,\left(\frac12-2\delta\right)r\right\}.
\]

Let \(I_0=\{2,\ldots,\lfloor\delta n\rfloor\}\). Expose the edges from \(1\) to
\(I_0\). Since \(\EE |N(1)\cap I_0|=(c\delta+o(1))\ln n\), Chernoff gives,
for \(a=c\delta/2\), w.h.p.
\[
    |N(1)\cap I_0|\ge a\ln n .
\]
On this event, put \(S_0=N(1)\cap I_0\), and observe that all vertices in
\(S_0\) are connected to \(1\) by a monotone edge. Set
\[
    M=\left\lfloor \frac{n}{K\ln n}\right\rfloor,
    \qquad
    T=\left\lfloor\left(\frac12-2\delta\right)K\ln n\right\rfloor .
\]
Let \(J_1,\ldots,J_T\) be consecutive intervals of length \(M\) immediately
after \(I_0\). As \(n\) is large, they all lie in
\(\{1,\ldots,\lfloor n/2\rfloor\}\), since
\[
    \lfloor\delta n\rfloor+TM
    \le \delta n+\left(\frac12-2\delta\right)n
    <\frac n2 .
\]

Let \(m=\lceil n^\tau\rceil\). We process the \(J_t\)'s in order, maintaining
sets \(S_t\) of vertices reachable from \(1\) by monotone paths contained in
\(I_0\cup J_1\cup\cdots\cup J_t\). Once \(|S_{t-1}|\ge m\), set
\(S_t=S_{t-1}\). Otherwise expose the edges between \(S_{t-1}\) and \(J_t\),
and set
\[
    S_t=S_{t-1}\cup\{y\in J_t: N(y)\cap S_{t-1}\ne\emptyset\}.
\]
Every vertex of \(S_t\) is reachable from \(1\) by a monotone path, since
\(J_t\) lies to the right of all previously used intervals.

Let \(\mathcal F_{t-1}\) be the history before step \(t\), and suppose
\(|S_{t-1}|=x<m\). Then \(x\ge a\ln n\). No edge from \(S_{t-1}\) to \(J_t\)
has previously been exposed, so conditional on \(\mathcal F_{t-1}\),
\[
    |S_t\setminus S_{t-1}|
    \sim \operatorname{Bin}\left(M,1-(1-p)^x\right).
\]
As \(\tau<1\), uniformly for \(x<m\) we have \(px=o(1)\), and therefore
\[
    \EE\bigl(|S_t\setminus S_{t-1}|\mid \mathcal F_{t-1}\bigr)
    =\left(\frac{c}{K}+o(1)\right)x .
\]
Since \(c_0<c\), Chernoff gives, for some constant \(\gamma>0\),
\[
    \PP\left(|S_t\setminus S_{t-1}|<\frac{c_0}{K}x
    \mid \mathcal F_{t-1}\right)
    \le e^{-\gamma x}
    \le n^{-\gamma a}.
\]
Call a step \(t\) active if \(|S_{t-1}|<m\), and bad if it is active and
\[
    |S_t\setminus S_{t-1}|<\frac{c_0}{K}|S_{t-1}|.
\]
For each \(t\), the conditional probability that \(t\) is bad, given
\(\mathcal F_{t-1}\), is at most \(n^{-\gamma a}\). Hence
\[
    \PP(\text{some bad step})\le Tn^{-\gamma a}=o(1).
\]
On the complementary event, every active step satisfies
\[
    |S_t|\ge \left(1+\frac{c_0}{K}\right)|S_{t-1}|.
\]
Thus, if the process has not reached \(m\) by time \(T\), then all \(T\) steps
were active and none was bad, so
\[
    |S_T|
    \ge a\ln n\left(1+\frac{c_0}{K}\right)^T
    =n^{(1/2-2\delta)r+o(1)}
    > n^\tau
\]
for all large \(n\). Since \(|S_T|\) is an integer, this implies
\(|S_T|\ge \lceil n^\tau\rceil=m\), a contradiction. Therefore, with
probability \(1-o(1)\), the process reaches size \(m\), and in particular
\(|S_T|\ge n^\tau\), proving the claim.
\end{proof}

\noindent To complete the proof, apply Claim~\ref{clm:one-sided-growth} in the left
half. W.h.p. there is a set
\[
    L\subseteq \{1,\ldots,\lfloor n/2\rfloor\},
    \qquad |L|\ge n^\tau,
\]
such that every \(u\in L\) is reachable from \(1\) by a monotone path contained
in the left half.

Applying the same argument to the right half in the reverse order, w.h.p.
there is a set
\[
    R\subseteq \{\lfloor n/2\rfloor+1,\ldots,n\},
    \qquad |R|\ge n^\tau,
\]
such that every \(v\in R\) has a monotone path from \(v\) to \(n\) contained in
the right half.

The two explorations expose only edges inside the two halves, so all edges
between \(L\) and \(R\) are still unexposed. Conditional on \(L\) and \(R\),
\[
    \PP(e(L,R)=0)
    \le (1-p)^{|L||R|}
    \le \exp(-p n^{2\tau})
    =o(1),
\]
because \(\tau>1/2\). Hence w.h.p. there is an edge \(uv\) with
\(u\in L\), \(v\in R\). Since \(u<v\), concatenating a monotone path from
\(1\) to \(u\), the edge \(uv\), and a monotone path from \(v\) to \(n\), gives
a monotone path from \(1\) to \(n\). This proves the supercritical assertion,
and hence the theorem.
\end{proof}

\end{document}